\documentclass[11pt,reqno]{amsart}

\usepackage{fullpage}
\usepackage{setspace}
\usepackage{color}
\usepackage{hyperref}  
\usepackage{graphicx}
\usepackage{datetime}
\usepackage{verbatim}
\usepackage{srcltx}

\usepackage{parskip}
\usepackage{amsfonts}
\usepackage{amsmath}

\def\EE{\mathbb E}

\def\cZ{\mathcal Z}

\def\EE{{\mathbb E}}

\def\({\left(}
\def\){\right)}  
\def\<{\langle}
\def\>{\rangle}
\let\:=\colon
\let\epsilon\varepsilon
\let\phi \varphi



\newtheoremstyle{note}
  {4pt}
  {4pt}
  {\sl}
  {}
  {\bfseries}
  {.}
  {.5em}
  {}

\newtheoremstyle{note2}
  {4pt}
  {4pt}
  {\normalfont}
  {}
  {\bfseries}
  {.}
  {.5em}
  {}

\newtheoremstyle{introthms}
  {3pt}
  {3pt}
  {\normalfont}
  {}
  {\slshape}
  {.}
  {.5em}
  {\thmnote{#3}}

\newtheoremstyle{cases}
  {2pt}
  {2pt}
  {\normalfont}
  {}
  {\slshape}
  {.}
  {.3em}
  {}

\theoremstyle{plain}
\newtheorem{theorem}              {Theorem}       [section]
\newtheorem{lemma}      [theorem] {Lemma}         
\newtheorem{corollary}  [theorem] {Corollary}     
\newtheorem{fact}       [theorem] {Fact}          

\theoremstyle{cases}

\theoremstyle{introthms}

\theoremstyle{note} 
\newtheorem{remark}    [theorem]   {Remark}       
\newtheorem{definition}[theorem]   {Definition}

\newtheorem{property}[theorem] {Property} 

\theoremstyle{note2}
\newtheorem*{acknowledge} {Acknowledgement}

\newcommand{\old}[1]{}

\usepackage{lineno}
\newcommand*\patchAmsMathEnvironmentForLineno[1]{%
\expandafter\let\csname old#1\expandafter\endcsname\csname #1\endcsname
\expandafter\let\csname oldend#1\expandafter\endcsname\csname end#1\endcsname
\renewenvironment{#1}%
{\linenomath\csname old#1\endcsname}%
{\csname oldend#1\endcsname\endlinenomath}}%
\newcommand*\patchBothAmsMathEnvironmentsForLineno[1]{%
\patchAmsMathEnvironmentForLineno{#1}%
\patchAmsMathEnvironmentForLineno{#1*}}%
\AtBeginDocument{%
\patchBothAmsMathEnvironmentsForLineno{equation}%
\patchBothAmsMathEnvironmentsForLineno{align}%
\patchBothAmsMathEnvironmentsForLineno{flalign}%
\patchBothAmsMathEnvironmentsForLineno{alignat}%
\patchBothAmsMathEnvironmentsForLineno{gather}%
\patchBothAmsMathEnvironmentsForLineno{multline}%
}

\title[Sidon sets]{%
  On Sidon sets in a random set of vectors}

\author[S.~J.~Lee]{Sang June Lee}
\address{Department of Mathematics, Duksung Women's University, South Korea} 
\email{sanglee242@duksung.ac.kr}

\thanks{
 E-mail: sanglee242@duksung.ac.kr \\
 The author was supported by Basic Science Research Program through the National Research Foundation of Korea (NRF) funded by the Ministry of Science, ICT \& Future Planning (2013R1A1A1059913).}

\date{\today, \currenttime}

\begin{document}
\singlespace
\pagestyle{plain}
\thispagestyle{empty}
\footskip=30pt

\shortdate
\settimeformat{ampmtime}

\begin{abstract}For positive integers $d$ and $n$, let $[n]^d$ be the set of all vectors $(a_1,a_2,\dots, a_d)$, where $a_i$ is an integer with $0\leq a_i\leq n-1$. A subset $S$ of $[n]^d$ is called a \emph{Sidon set} if all sums of two (not necessarily distinct) vectors in $S$ are distinct.

In this paper, we estimate two numbers related to the maximum size of Sidon sets in $[n]^d$. First, let $\cZ_{n,d}$ be the number of all Sidon sets in $[n]^d$.
We show that $\log (\cZ_{n,d})=\Theta(n^{d/2})$, where the constants of $\Theta$ depend only on $d$. Next, we estimate the maximum size of Sidon sets contained in a random set $[n]^d_p$, where $[n]^d_p$ denotes a random set obtained from $[n]^d$ by choosing each element independently with probability $p$.

\end{abstract}

\maketitle
\onehalfspace

\section{Introduction}

For positive integers $d$ and $n$, let $[n]^d$ be the set of all vectors $(a_1,a_2,\dots, a_d)$, where $a_i$'s are integers with $0\leq a_i\leq n-1$. A subset $S$ of $[n]^d$ is called a \emph{Sidon set} if all sums of two (not necessarily distinct) vectors in $S$ are distinct. 
A well-known problem on Sidon sets in $[n]^d$ is the determination of the maximum size $F([n]^d)$ of Sidon sets in $[n]^d$. For $d=1$, Erd\H{o}s and Tur\'an~\cite{erdoes41:_probl_of_sidon} showed in 1941  that $F([n])\leq n^{1/2}+O(n^{1/4})$. Then, Lindstr\"om~\cite{Lindstrom1969}, in 1969, improved the bound to $F([n])\leq n^{1/2}+n^{1/4}+1$. On the other hand, in 1944, Chowla~\cite{chowla44:_solut_of_probl_of_erdoes} and Erd\H{o}s~\cite{erdoes44:_probl_of_sidon}  observed that a result of Singer~\cite{Si} implies that $F([n])\geq n^{1/2}-O(n^{5/16})$. Consequently, we know $F([n])=n^{1/2}(1+o(1))$. For a general $d\geq 1$, Lindstr\"om~\cite{Lindstrom1972} showed in 1972 that $F([n]^d)\leq n^{d/2}+O(n^{d^2/(2d+2)})$. On the other hand, in 2010, Cilleruelo~\cite{Cilleruelo2010} proved that $F([n]^d) \geq F([n^d])\geq n^{d/2}-O(n^{5d/16}).$ Therefore,
\begin{equation}\label{eq:F}F([n]^d)=n^{d/2}(1+o(1)).\end{equation} For more information, see the classical monograph of Halberstam and Roth~\cite{HalberstamRoth2ndEd} and a survey paper by O'Bryant~\cite{obryant04:_biblio_Sidon}.

In this paper we consider two numbers related to the number $F([n]^d)$. The first one is  the number $\cZ_{n,d}$ of all Sidon sets contained in $[n]^d$. The second one is the maximum size of Sidon sets contained in a random subset of $[n]^d$ instead of $[n]^d$.

We first start with the problem of estimating $\cZ_{n,d}$. Recalling that $F([n]^d)=n^{d/2}(1+o(1))$, one can easily see that
$$2^{F([n]^d)}\leq \cZ_{n,d}\leq \sum_{k=1}^{F([n]^d)} \binom{n^d}{k}\leq F([n]^d)\binom{n^d}{F([n]^d)}.$$
This implies the following.
\begin{fact}\label{fact:Z}
$$ 2^{n^{d/2}(1+o(1))}\leq \cZ_{n,d}\leq n^{(d/2)n^{d/2}(1+o(1))}.$$
\end{fact}
In this paper, we improve the above upper bound as follows.
\begin{theorem}\label{thm:main1} For a positive integer $d$, there exists a positive constant $c=c(d)$ such that, for any sufficiently large $n=n(d)$,
$$\cZ_{n,d}\leq 2^{cn^{d/2}}.$$
\end{theorem}

Note that this upper bound matches the lower bound in Fact~\ref{fact:Z} up to a multiplicative constant factor in the exponent. Our proof of Theorem~\ref{thm:main1} will be provided in Subsection~\ref{sec:number of Sidon sets}. The case $d=1$ of Thereom~\ref{thm:main1} was also proved in~\cite{sidon}.

Next, we deal with the maximum size of Sidon sets contained in a random subset of $[n]^d$. Let $[n]^d_p$ be a random set obtained from $[n]^d$ by choosing each element independently with probability $p$. Let $F([n]^d_p)$ be the maximum size of Sidon sets in a random set $[n]^d_p$. Our result about $F([n]^d_p)$ is as follows.

\begin{theorem}\label{thm:main2}
  For a positive integer $d$, let $a$ be a constant with $-d< a\leq 0$, and let $p=p(n)=n^a(1+o(1))$.  Then, there exists a constant $b=b(a)$ such
  that, asymptotically almost surely (a.a.s.), that is, with probability tending to $1$ as $n\rightarrow \infty$,
  \begin{equation}
    \label{eq:b_def}
    F([n]^d_p)=n^{b+o(1)}.
  \end{equation}
  Moreover,
  \begin{equation}
    \label{eq:b(a)}
    b(a)=
    \begin{cases}
     a+d    &\text{if }\,\,-d< a\leq -2d/3,\\
      d/3  &\text{if }\,\,-2d/3 \leq a\leq -d/3,\\
      (a+d)/2  &\text{if }\,\,-d/3\leq a\leq 0.
    \end{cases}
  \end{equation}
\end{theorem}

\begin{figure}
\begin{center}
\includegraphics[scale=0.2]{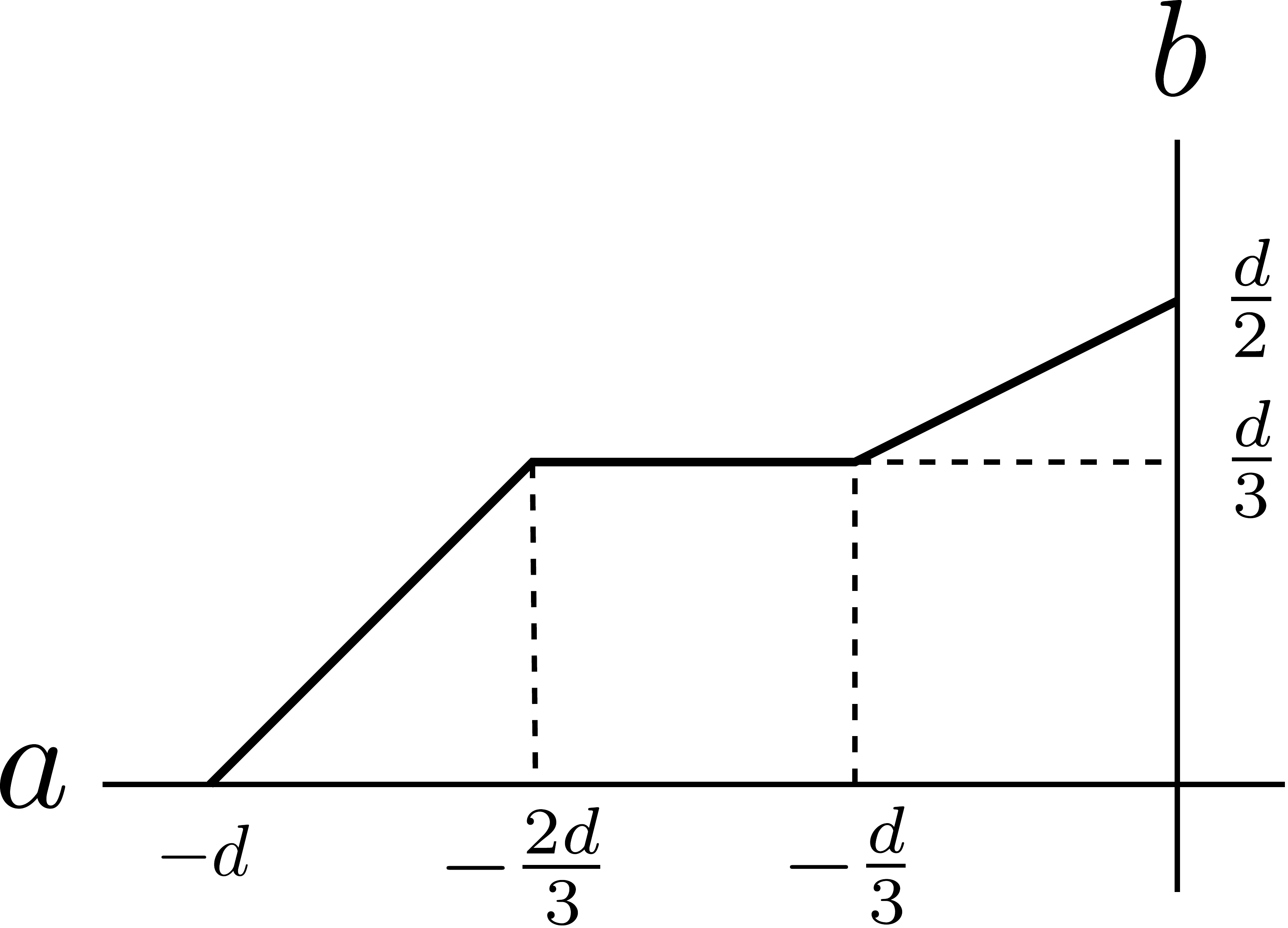}
\end{center}
\caption{The graph of $b=b(a)$ in Theorem~\ref{thm:main2}}
\label{fig:graph}
\end{figure}

The graph of $b=b(a)$ is given in Figure~\ref{fig:graph}. A refined version of Theorem~\ref{thm:main2} is stated in Theorems~\ref{thm:small}--\ref{thm:large} in Subsection~\ref{sec:random}. Theorems~\ref{thm:small}--\ref{thm:large} will be proved in Sections~\ref{sec:upper bounds} and~\ref{sec:lower bounds}.
The case $d=1$ of Theorem~\ref{thm:main2} was also proved in~\cite{sidon}. 

\subsection{Remark and Notation}

From now on, let $d$ be a fixed positive integer. Constants in $O, \Omega$, and $\Theta$ may depend on $d$.  We write $f=o(g)$ if $f/g$ goes to $0$ as $n\rightarrow \infty$. We also write $f\ll g$ if $f/g=o(1)$.

\section{Main Results}\label{sec:main results}

\subsection{The number of Sidon sets of a given size}

We will obtain an upper bound on the number of Sidon sets in $[n]^d$ of a given size. For a positive integer $t$, let $\cZ_{n,d}(t)$ be the number of Sidon sets in $[n]^d$ of size $t$.  Observe that the following result applies when $t=\Omega\(n^{d/3}(\log n)^{1/3}\)$.

\begin{lemma}\label{lem:Z_t} Let $d$ be a positive integer. For a sufficiently large integer $n=n(d)$, the following holds:
If  $t$ is a positive integer with $t\geq 2s_0$, where $s_0=(d2^{d+1})^{1/3} n^{d/3} (\log n)^{1/3}$, then 
\begin{equation*}
\cZ_{n,d}(t)\leq n^{2(d+1)s_0}\({e2^{d+5}n^d \over t^2}\)^{t}.
\end{equation*}
\end{lemma}

Our proof of Lemma~\ref{lem:Z_t} will be given in Subsection~\ref{sec:large}. 
Lemma~\ref{lem:Z_t} will be used in order to prove Theorem~\ref{thm:main1} (see Subsection~\ref{sec:number of Sidon sets} for its proof) and the upper bounds in Theorems~\ref{thm:middle_large} and~\ref{thm:large} (see Section~\ref{sec:upper bounds} for its proof). 

The next lemma provides an upper bound on the number $\cZ_{n,d}(t)$ for $t=\Omega(n^{d/3})$.
Observe that the range of $t$ here is a bit wider than the range of $t$ in Lemma~\ref{lem:Z_t}. 

\begin{lemma}\label{lem:Z_s+q+r(2)} Let $\gamma$ and $\omega$ be real numbers and let $n$, $s^*$ and $t$  be positive integers satisfying that
\begin{equation}\label{eq:q(2)} 0<\gamma<s^*/2^{d+1},  \hskip 1em s^*=  2^{(d+1)/3}n^{d/3} (\log \gamma)^{1/3}, \end{equation}
\begin{equation}\label{eq:q(3)}  \omega\geq 4, \hskip 2em \mbox{ and } \hskip 2em t=\omega s^*. \end{equation}
Then, 
\begin{equation*} \cZ_{n,d}(t)\leq \(\frac{4en^d}{t\gamma^{1-2/\omega}}\)^t. \end{equation*}
\end{lemma}

\begin{remark}\label{rem:number} For $d=1$, a version of Lemma~\ref{lem:Z_s+q+r(2)} was given in Lemma 3.3 in~\cite{sidon}, but we improve the previous one as follows: 
\begin{enumerate}\item We have a better upper bound on $\cZ_{n,1}(t)$ by removing the multiplicative factor $\omega$ in the base in Lemma 3.3 of~\cite{sidon}. \item We remove the variable $\sigma$ used in Lemma 3.3 of~\cite{sidon}. \end{enumerate}
\end{remark}

Our proof of Lemma~\ref{lem:Z_s+q+r(2)} will be given in Subsection~\ref{sec:small}. Lemma~\ref{lem:Z_s+q+r(2)} will be applied to our proof of the upper bound in Theorem~\ref{thm:middle}. (See Section~\ref{sec:upper bounds} for the proof.)

\subsection{The maximum size of Sidon sets in a random set $[n]^d_p$}\label{sec:random}

Recall that $[n]^d_p$ is a random set obtained  from $[n]^d$ by choosing each element independently with probability $p$. Also, recall that $F\([n]^d_p\)$ denotes the maximum size of Sidon sets in a random set $[n]^d_p$. We state our results on the upper and lower bounds of $F\([n]^d_p\)$ in Theorems~\ref{thm:small}--\ref{thm:large} in full. Recall that $f\ll g$ if $f/g=o(1)$.

\begin{theorem}\label{thm:small}
The following holds a.a.s.:\\
If \hskip 0.5em $n^{-d}\ll p\ll n^{-2d/3}$, then 
\begin{equation}\label{eq:small1} F\([n]^d_p\)= \(1+o(1)\)n^dp.\end{equation}
If \hskip 0.5em $ \displaystyle n^{-d}\ll p\leq2n^{-2d/3}$, then
\begin{equation}\label{eq:small2} \(1/3+o(1)\)n^dp \leq F\([n]^d_p\)\leq \(1+o(1)\)n^dp.\end{equation}
 \end{theorem}

\begin{theorem}\label{thm:middle} Let  $\epsilon<d/3$.
If \hskip 0.5em $2 n^{-2d/3}\leq p\leq n^{-d/3-\epsilon}$, then there exist a positive absolute constant $c_1$ and a positive constant $c_2=c_2(d)$ such that a.a.s. 
\begin{equation*}
c_1 n^{d/3}\(\log (n^{2d}p^3)\)^{1/3}\leq  F\([n]^d_p\)\leq c_2 n^{d/3}\(\log (n^{2d}p^3)\)^{1/3}.\end{equation*}
\end{theorem}

\begin{theorem}\label{thm:middle_large}  Let  $\epsilon<d/3$. 
If \hskip 0.5em $n^{-d/3-\epsilon}\leq p\leq n^{-d/3}\(\log n\)^{8/3}$, then there exist a positive absolute constant $c_3$ and a positive constant $c_4=c_4(d)$ such that a.a.s. \begin{equation*}
c_3 n^{d/3}\(\log n\)^{1/3}\leq  F\([n]^d_p\)\leq c_4 n^{d/3}\(\log n\)^{4/3}.\end{equation*}
\end{theorem}

\begin{theorem}\label{thm:large}
If \hskip 0.5em $n^{-d/3}\(\log n\)^{8/3}\leq p\leq 1$, then there exist a positive absolute constant $c_5$ and a positive constant $c_6=c_6(d)$ such that a.a.s. \begin{equation*}
c_5n^{d/2}p^{1/2} \leq F\([n]^d_p\) \leq c_6n^{d/2}p^{1/2}.
\end{equation*}

\end{theorem}

\subsection{Organization}

In Section~\ref{sec:Sidon}, we prove Theorem~\ref{thm:main1} and  Lemmas~\ref{lem:Z_t} and~\ref{lem:Z_s+q+r(2)}. Our proof of the upper bounds in Theorems~\ref{thm:small}--\ref{thm:large} will be provided in Section~\ref{sec:upper bounds}. In Section~\ref{sec:lower bounds}, we prove the lower bounds in Theorems~\ref{thm:small}--\ref{thm:large}.

\section{The number of Sidon sets in $[n]^d$ of a given size}\label{sec:Sidon}

\subsection{The number of Sidon sets}\label{sec:number of Sidon sets}

Now we show Theorem~\ref{thm:main1} by using Lemma~\ref{lem:Z_t}.

\begin{proof}[\textbf{Proof of Theorem~\ref{thm:main1}}] 
We have that
\begin{eqnarray*}
\cZ_{n,d} &=& \sum_{t=1}^{n^d} \cZ_{n,d}(t) = \sum_{t=1}^{F([n]^d)} \cZ_{n,d}(t),
\end{eqnarray*}
where the second equality holds since $F([n]^d)$ is the maximum size of Sidon sets in $[n]^d$. Since $F([n]^d)=n^{d/2}(1+o(1))$, we have that
\begin{eqnarray}\label{eq:sum_Z}
\cZ_{n,d} = \sum_{t=1}^{n^{d/3}\log n} \cZ_{n,d}(t)+\sum_{t=n^{d/3}\log n+1}^{F([n]^d)} \cZ_{n,d}(t).
\end{eqnarray}

The first sum of~\eqref{eq:sum_Z} is estimated by 
\begin{eqnarray*}
\sum_{t=1}^{n^{d/3}\log n} \cZ_{n,d}(t) &\leq & \sum_{t=1}^{n^{d/3}\log n} {n^d \choose t} \leq n^{d/3}\log n\cdot \(\frac{en^d}{n^{d/3}\log n}\)^{n^{d/3}\log n} \\
&\leq & n^{(2d/3)n^{d/3}\log n(1+o(1))}\leq 2^{c_1n^{d/3}(\log n)^2},
\end{eqnarray*}
where $c_1=c_1(d)$ is a positive constant depending only on $d$. Next, it follows from Lemma~\ref{lem:Z_t} that the second sum of~\eqref{eq:sum_Z} is estimated by
\begin{eqnarray*}
\sum_{t=n^{d/3}\log n+1}^{F([n]^d)} \cZ_{n,d}(t) &\leq & F([n]^d)\cdot n^{c_2 n^{d/3}(\log n)^{1/3}}\(\frac{c_3 n^d}{(F([n]^d))^2}\)^{F([n]^d)} \\
&\leq & 2^{c_4 n^{d/2}},
\end{eqnarray*}
where $c_2, c_3$, and $c_4$ are positive constants depending only on $d$.
Therefore, in view of identity~\eqref{eq:sum_Z}, the above estimates of the first and second sums of~\eqref{eq:sum_Z} imply
Theorem~\ref{thm:main1}.
\end{proof}

\subsection{The number of Sidon sets of a larger size}\label{sec:large}

Recall that $\cZ_{n,d}(t)$ is the number of Sidon sets in $[n]^d$ of size $t$.
Now we show Lemma~\ref{lem:Z_t} which gives an upper bound on $\cZ_{n,d}(t)$ for $t=\Omega\(n^{d/3}(\log n)^{1/3}\)$. Our proof uses the following strategy from~\cite{sidon}.  Let $s$ be an integer with $s<t$, and let $S$ be a seed Sidon set in $[n]^d$ of size $s$. For such a Sidon set $S$, we estimate the number of extensions of $S$ to larger Sidon sets $S^*$ of size $t$ containing $S$. Then, by summing over all Sidon sets $S$ of size $s$, we will obtain an upper bound on $\cZ_{n,d}(t)$.
In order to bound the number of extensions, we define the following graph. 
\begin{definition}\label{def:G_S}For a Sidon set $S$ in $[n]^d$, let $G_S$ be the graph on $V=[n]^d\setminus S$ in which $\{v_1, v_2\}$ is an edge of $G_S$ if and only if there exist some $b_1,b_2\in S$ such that $v_1+b_1=v_2+b_2$. 
\end{definition}
Observe that if $S^*$ is a Sidon set in $[n]^d$ of size $t$ containing $S$, then the set $S^*\setminus S$ is an independent set in  $G_S$ of size $t-s$. Hence, the number of extensions of $S$ to larger Sidon sets $S^*$ of size $t$ is bounded above by the number of independent sets in $G_S$ of size $t-s$.

In order to bound the number of independent sets in $G_S$ of a given size, we will use the following result from~\cite{sidon}.

\begin{lemma}[Lemma 3.1 of~\cite{sidon}]\label{lem:locally dense} 
For positive integers $N$ and $R$ and a positive real number $\beta$, let $G$ be a graph on $N$ vertices such that for every vertex set $U$ with $|U|\geq R$, the number $e(U)$ of edges  in the subgraph of $G$ induced on $U$ satisfies \begin{equation}\label{eq:locally dense_cond1} e(U)\geq \beta{|U| \choose 2}. \end{equation}
If  $q$ is a positive integer  satisfying \begin{equation}\label{eq:locally dense_cond2}  q\geq \beta^{-1}\log(N/R),\end{equation} 
then, for all positive integers $r$, the number of independent sets in $G$ of size $q+r$ is at most
\begin{equation} {N\choose q}{R \choose r}.
\end{equation}
\end{lemma}

Next we show that the graph $G_S$ with a Sidon set $S$ satisfies condition~\eqref{eq:locally dense_cond1} with suitable $R$ and $\beta$.

\begin{lemma}\label{lem:G_S} For a Sidon set $S$ in $[n]^d$ of size $s$, the graph $G_S$ on $N:=n^d-s
$ vertices satisfies the following: 
For every vertex set $U$ with \begin{equation}\label{eq:U}|U|\geq (2^{d+1}/s) n^d,\end{equation} the number $e(U)$ of edges  in the subgraph of $G_S$ induced on $U$ satisfies
\begin{equation}\label{eq:G_S} e(U)\geq \frac{s^2}{2^{d+1}n^d}{|U| \choose 2}.
\end{equation}
\end{lemma}

\begin{proof} Let $U$ be an arbitrary vertex set of $G_S$ with $|U|\geq (2^{d+1}/s)n^d$. We define an auxiliary bipartite graph $B$ with disjoint vertex classes $[2n]^d$ and $U$ in which a vertex $w\in [2n]^d$ is adjacent to a vertex $u\in U$ if and only if there exists $b\in S$ such that $w=u+b$. Observe that distinct vertices $u_1$ and $u_2$ in $U$  have a common neighbor $w\in [2n]^d$ if and only if $\{u_1, u_2\}$ is an edge of the subgraph $G_S[U]$ of $G_S$ induced on $U$. Hence, we infer that $ e(U)\leq \sum_{w\in [2n]^d}{ {d_B(w)}\choose 2}$, where $d_B(w)$ denotes the degree of $w$ in~$B$.

Now we claim that \begin{equation}\label{eq:e(U)} e(U)= \sum_{w\in [2n]^d}{ {d_B(w)}\choose 2}.\end{equation} In order to prove~\eqref{eq:e(U)}, we need to show that $B$ contains no $4$-cycle, i.e., that two distinct vertices in $U$ do not have two distinct common neighbors in $[2n]^d$. Towards contradiction, suppose that there is a $4$-cycle in $B$, that is,  both $u_1$ and $u_2$ $(u_1\neq u_2)$ in $U$ are adjacent to both $w_1$ and $w_2$ $(w_1\neq w_2)$ in $[2n]^d$. From the definition of $B$, there exist some $b_{11}, b_{12}, b_{21}, b_{22}\in S$ such that $w_1=u_1+b_{11}$, $w_1=u_2+b_{12}$, $w_2=u_1+b_{21}$ and $w_2=u_2+b_{22}$. Thus, $u_1+b_{11}=u_2+b_{12}$ and $u_1+b_{21}=u_2+b_{22}$, and hence, we have that $b_{11}-b_{21}=b_{12}-b_{22}$, that is, $b_{11}+b_{22}=b_{12}+b_{21}$. Since $S$ is a Sidon set, we infer that $\{b_{11},b_{22}\}=\{b_{12},b_{21}\}$. However, by the assumptions $u_1\neq u_2$ and $w_1\neq w_2$, we have that $b_{11}\neq b_{12}$ and $b_{11}\neq b_{21}$, which contradicts to $\{b_{11},b_{22}\}=\{b_{12},b_{21}\}$. Therefore, there is no 4-cycle in $B$, and hence, identity~\eqref{eq:e(U)} holds.

It follows from~\eqref{eq:e(U)} that
 \begin{eqnarray*}
e(U)=\sum_{w\in [2n]^d}{{d_B(w)} \choose 2 }\geq (2n)^d {\frac{1}{(2n)^d}\sum_{w\in [2n]^d} d_B(w) \choose 2},
\end{eqnarray*}
where the inequality follows from the convexity of $\binom{x}{2}$.
Since $B$ is a bipartite graph in which $d_B(u)=s$ for all $u\in U$, we have that
\begin{eqnarray*}
e(U) &=& (2n)^d {\frac{1}{(2n)^d}\sum_{u\in U} d_B(u) \choose 2} 
= (2n)^d {s|U|/(2n)^d \choose 2}\\&=&(2n)^d\cdot \frac{1}{2}\frac{s|U|}{(2n)^d}\(\frac{s|U|}{(2n)^d}-1\).
\end{eqnarray*}
Under the assumption~\eqref{eq:U}, that is, $1\leq \frac{1}{2}\frac{s|U|}{(2n)^d}$, we infer that
\begin{eqnarray*}
e(U) \geq \frac{s|U|}{2}\cdot \frac{1}{2}\frac{s|U|}{(2n)^d} = \frac{s^2}{2^{d+1}n^d}\frac{|U|^2}{2} \geq  \frac{s^2}{2^{d+1}n^d}{|U| \choose 2}.
\end{eqnarray*} 
This completes the proof of Lemma~\ref{lem:G_S}.
\end{proof}

Now we are ready to bound the number of Sidon sets of a larger size by applying Lemmas~\ref{lem:locally dense} and~\ref{lem:G_S} as follows.
\begin{lemma}\label{lem:Z_s+q+r} Let $n$, $s$, and $q$ be positive integers satisfying
\begin{equation}\label{eq:q} s^2q\geq d2^{d+1}n^d \log n. \end{equation}
Then, for any integer $r\geq 0$, we have
\begin{equation}\label{eq:Z_s+q+r} \cZ_{n,d}(s+q+r)\leq \cZ_{n,d}(s){n^d \choose q}{2^{d+1}n^d/s \choose r}.\end{equation}
\end{lemma}

\begin{proof} Fix $S$ as an arbitrary Sidon set in $[n]^d$ of size $s$. We first consider the number of Sidon sets $S^*$ of size $s+q+r$ containing $S$.
Recall that if $S^*$ is a Sidon set of size $s+k$ containing $S$, then the set $S^*\setminus S$ is an independent set in  $G_S$ of size $k$. Hence, in order to bound the number of Sidon sets of size $s+q+r$ containing $S$, we are going to estimate the number of independent sets of size $q+r$ in $G_S$. To this end, we will apply Lemma~\ref{lem:locally dense} to the graph $G_S$. 

We first check conditions~\eqref{eq:locally dense_cond1} and~\eqref{eq:locally dense_cond2} of Lemma~\ref{lem:locally dense}. First, Lemma~\ref{lem:G_S} implies that~\eqref{eq:locally dense_cond1} holds with $R= (2^{d+1}/s)n^d$ and $\beta=s^2/(2^{d+1}n^d)$. Next,  condition~\eqref{eq:locally dense_cond2} follows from ineqality~\eqref{eq:q}  because 
\begin{eqnarray*} q&\geq&  \frac{d2^{d+1}n^d\log n}{s^2} = \frac{2^{d+1}n^d\log (n^d)}{s^2} \geq\beta^{-1}\log\(\frac{N}{R}\).\end{eqnarray*}

Thus, Lemma~\ref{lem:locally dense} with $G=G_S$, $N\leq n^d$ and $R= (2^{d+1}/s)n^d$ gives that for any integer $r\geq0$, the number of independent sets in $G_S$ of size $q+r$ is at most $ {n^d \choose q}{2^{d+1}n^d/s \choose r}.$ Consequently, the number of Sidon sets of size $s+q+r$ containing $S$ is at most $ {n^d \choose q}{2^{d+1}n^d/s \choose r}.$ 
By summing over all Sidon sets $S$ of size $s$, we infer that the number of all Sidon sets of size $s+q+r$ is at most $\cZ_{n,d}(s){n^d \choose q}{2^{d+1}n^d/s \choose r},$ which completes the proof of Lemma~\ref{lem:Z_s+q+r}. 
\end{proof}

Now we show Lemma~\ref{lem:Z_t} by applying Lemma~\ref{lem:Z_s+q+r} iteratively.

\begin{proof}[\textbf{Proof of Lemma~\ref{lem:Z_t}}]
Since $F([n]^d)=n^{d/2}(1+o(1))$, we infer  that $\cZ_{n,d}(t)=0$ if $t>1.1n^{d/2}$ for a sufficiently large $n=n(d)$, depending only on $d$. 
Hence, let $t$ be an integer satisfying \begin{equation}\label{eq:range of t}2s_0\leq t \leq 1.1 n^{d/2},\end{equation} where $s_0=(d2^{d+1} n^d \log n)^{1/3}$. Let $K$ be the largest integer satisfying $t2^{-K}\geq s_0$. We define three sequences $s_k$, $q_k$, and $r_k$ as follows: 
for $1\leq k\leq K+1$, \begin{equation}\label{eq:s,q,r} s_k=2s_{k-1}=t2^{-K+k-1}, \hskip 1em q_k=q_{k-1}/4=s_0/4^{k-1}, \hskip 1em and \hskip 1em r_k=s_{k+1}-s_k-q_k.
\end{equation} 
Under the assumption $s_0=(d2^{d+1} n^d \log n)^{1/3}$ and the definition of $s_k=2s_{k-1}$ and $q_k=q_{k-1}/4$, we have that, for $1\leq k\leq K$, $$s_k^2q_k=s_1^2q_1\geq s_0^3=d2^{d+1} n^d \log n.$$ Equivalently, condition~\eqref{eq:q} holds with $s=s_k$ and $q=q_k$. Thus,  
Lemma~\ref{lem:Z_s+q+r} with $s=s_k$, $q=q_k$, and $r=r_k$ gives that for $1\leq k \leq K$,
\begin{equation*} \cZ_{n,d}(s_{k+1})=\cZ_{n,d}(s_k+q_k+r_k)\leq \cZ_{n,d}(s_k){n^d \choose q_k}{2^{d+1}n^d/s_k \choose r_k}.\end{equation*}
Consequently,
\begin{equation}\label{eq:t}
\cZ_{n,d}(t)=\cZ_{n,d}(s_{K+1})\leq {n \choose s_1}\cdot \prod_{k=1}^{K}{n^d \choose q_k}\cdot \prod_{k=1}^{K}{2^{d+1}n^d/s_k \choose r_k}.
\end{equation} 

Now we estimate three parts of the right-hand side of~\eqref{eq:t} separately. The first part is estimated by
 \begin{equation}\label{eq:z_t(1)} {n \choose s_1 }\leq {n \choose 2s_0} \leq n^{2s_0}.
\end{equation} 
Next, for the second part of~\eqref{eq:t}, we have that
 \begin{eqnarray}\label{eq:z_t(2)} \prod_{k=1}^{K}{n^d \choose q_k} &\leq &\prod_{k=1}^{K}(n^d)^{q_k} = (n^d)^{\sum_{k=1}^{K} q_k}\overset{\eqref{eq:s,q,r} }{=}(n^d)^{s_0\sum_{k=1}^{K}4^{-k+1}}\nonumber\\&\leq& (n^d)^{s_0 (4/3)}\leq n^{2ds_0},
\end{eqnarray} where the second inequality follows from $\sum_{k=1}^{K}4^{-k+1}\leq \sum_{k=1}^{\infty}4^{-k+1}=4/3.$ For the last part of~\eqref{eq:t}, we first have that
 \begin{eqnarray*}
\prod_{k=1}^{K}{2^{d+1}n^d/s_k \choose r_k}&\leq & \prod_{k=1}^{K}{2^{d+1}n^d/s_k \choose r_k+q_k}
\end{eqnarray*}
 since
\begin{eqnarray*}
\frac{r_k+q_k}{2^{d+1}n^d/s_k}\leq \frac{s_{k+1}-s_k}{2^{d+1}n^d/s_k}=\frac{s_k}{2^{d+1}n^d/s_k}=\frac{s_k^2}{2^{d+1}n^d}\leq \frac{t^2}{2^{d+1}n^d}\overset{\eqref{eq:range of t}}{\leq} \frac{1}{2}.
\end{eqnarray*}
We further have that
\begin{eqnarray}\label{eq:z_t(3)}\prod_{k=1}^{K}{2^{d+1}n^d/s_k \choose r_k}&\leq &\prod_{k=1}^{K}{2^{d+1}n^d/s_k \choose s_{k+1}-s_k}=\prod_{k=1}^{K}{2^{d+1}n^d/s_k \choose s_k} \nonumber \\
&\leq & \prod_{k=1}^{K}\({e2^{d+1}n^d \over s_k^2}\)^{s_k} 
= \prod_{k=1}^{K}\({e2^{d+1}n^d \over s_{K-k+1}^2}\)^{s_{K-k+1}} \nonumber \\ &\overset{\eqref{eq:s,q,r} }{=}&\prod_{k=1}^{K}\({e2^{d+1}2^{2k}n^d \over t^2}\)^{t2^{-k}} \nonumber \\ &=&\({e2^{d+1}n^d \over t^2}\)^{t\sum_{k=1}^{K}2^{-k}}2^{2t\sum_{k=1}^{K}k2^{-k}} \nonumber \\
&\leq & \({e2^{d+1}n^d \over t^2}\)^{t}2^{4t} =\({e2^{d+5}n^d \over t^2}\)^{t}. \end{eqnarray}

In view of~\eqref{eq:t}, combining~\eqref{eq:z_t(1)}--\eqref{eq:z_t(3)} yields that
$\cZ_{n,d}(t)\leq n^{2(d+1)s_0}\({e2^{d+5}n^d \over t^2}\)^{t}$ for $2s_0\leq t \leq 1.1 n^{d/2}$, which completes our proof of Lemma~\ref{lem:Z_t}.
\end{proof}

\subsection{The number of Sidon sets of a smaller size}\label{sec:small}

Now we show Lemma~\ref{lem:Z_s+q+r(2)} which gives an upper bound on $\cZ_{n,d}(t)$ for $t=\Omega(n^{d/3})$. Our proof of Lemma~\ref{lem:Z_s+q+r(2)} is similar to the proof of Lemma~\ref{lem:Z_s+q+r}, and hence, we only give a sketch.
By weakening condition~\eqref{eq:U} of Lemma~\ref{lem:G_S} into $|U|\geq  n^d/\gamma$, where $1/\gamma\geq 2^{d+1}/s$, we clearly have the following corollary of Lemma~\ref{lem:G_S}. (We omit the proof.)

\begin{corollary}\label{lem:G_S(2)} Let $\gamma$ be an arbitrary real number with $0<\gamma\leq s/2^{d+1}$. For a Sidon set $S$ in $[n]^d$ of size $s$, the graph $G_S$ on $N:=n^d-s$ vertices satisfies the following: 
For every vertex set $U$ with $|U|\geq  n^d/\gamma,$ the number $e(U)$ of edges in the subgraph of $G_S$ induced on $U$ satisfies
\begin{equation}\label{eq:G_S(2)} e(U)\geq \frac{s^2}{2^{d+1}n^d}{|U| \choose 2}.
\end{equation}
\end{corollary}

Combining Lemma~\ref{lem:locally dense} and Corollary~\ref{lem:G_S(2)} implies Lemma~\ref{lem:Z_s+q+r(2)} as follows.

\begin{proof}[\textbf{Proof of Lemma~\ref{lem:Z_s+q+r(2)}}] Before applying Lemma~\ref{lem:locally dense} with $G=G_S$,  we first check conditions~\eqref{eq:locally dense_cond1} and~\eqref{eq:locally dense_cond2} in Lemma~\ref{lem:locally dense} with $G=G_S$.   
First, by Corollary~\ref{lem:G_S(2)},  the graph $G_S$ satisfies~\eqref{eq:locally dense_cond1} with $R= n^d/\gamma$ and $\beta=s^2/(2^{d+1}n^d)$. Next,  condition~\eqref{eq:locally dense_cond2} 
 holds by setting $s=q=s^*$ where $s^*=\(2^{(d+1)}n^{d}\log \gamma\)^{1/3}$.

Now Lemma~\ref{lem:locally dense} with $G=G_S$ and $s=q=s^*$ implies that, for $t\geq 4s^*$, 
\begin{eqnarray}\label{eq:Z_s+q+r(10)} \cZ_{n,d}(t)&\leq& \cZ_{n,d}(s^*){n^d \choose s^*}{n^d/\gamma \choose t-2s^*} \leq {n^d \choose s^*}{n^d \choose s^*}{n^d/\gamma \choose t-2s^*} \nonumber \\
&\leq & \({en^d \over s^*}\)^{2s^*} \({en^d \over \gamma(t-2s^*)  }\)^{t-2s^*}=\(\frac{en^d}{s^*} \)^t \(\frac{1}{\gamma(\omega-2)}\)^{t-2s^*} \nonumber\\
&=& \(\frac{e\omega n^d}{t}\)^t \(\frac{1}{\gamma(\omega-2)} \)^{t(1-2/\omega)} = \(\frac{e\omega n^d}{t[\gamma(\omega-2)]^{1-2/\omega}} \)^t \nonumber \\
&=& \(C_{\omega} \frac{n^d}{t \gamma^{1-2/\omega}} \)^t,
\end{eqnarray}
where $C_{\omega}=e\omega/ (\omega-2)^{1-2/\omega}$.
We also have that
\begin{eqnarray}\label{eq:C_omega}
C_{\omega}=\frac{e\omega}{(\omega-2)^{1-2/\omega}}\leq \frac{e\omega}{(\omega/2)^{1-2/\omega}}=e 2^{1-2/\omega} \omega^{2/\omega} \leq 2e\omega^{2/\omega}\leq 2e 4^{1/2} = 4e,
\end{eqnarray} where the first inequality follows from the assumption $\omega\geq 4$, and the last inequality follows from the fact that $f(x)=x^{2/x}$ is a decreasing function for $x\geq 4$.
Combining~\eqref{eq:Z_s+q+r(10)} and~\eqref{eq:C_omega} completes our proof of Lemma~\ref{lem:Z_s+q+r(2)}.
\end{proof}


\section{Upper bounds on $F([n]_p)$}\label{sec:upper bounds}

In this section we prove the upper bounds in Theorems~\ref{thm:small}--\ref{thm:large}. We first provide our proof of the upper bound in Theorem~\ref{thm:small}.
In the proof, we will use the following version of Chernoff's bound.

\begin{lemma}[\textbf{Chernoff's bound}, Corollary 4.6 in~\cite{MU2005}] \label{lem:Chernoff} Let $X_i$ be independent random variables such that $\Pr[X_i=1]=p_i$ and $\Pr[X_i=0]=1-p_i$, and let $X=\sum_{i=1}^{n} X_i$.  
For $0<\lambda<1$,
\begin{equation*}\Pr \Big[|X-\EE(X)|\geq \lambda\EE(X)\Big]\leq 2 \exp\Big(-\frac{\lambda^2}{3}\EE(X)\Big). \end{equation*}
\end{lemma}

\begin{proof}[\textbf{Proof of the upper bound in Theorem~\ref{thm:small}}] We clearly have that
$F\([n]^d_p\)\leq \big|[n]^d_p\big|.$ Hence, in order to show the upper bound in Theorem~\ref{thm:small}, it suffices to show that a.a.s. \begin{equation}\label{eq:X}X:=\big|[n]^d_p\big|\leq n^dp(1+o(1)). \end{equation}

By the definition of $[n]^d_p$, we have that the expectation $\EE(X)$ is $n^d p$. Then,  Lemma~\ref{lem:Chernoff} implies that a.a.s. $X=n^dp(1+o(1))$ provided that $p\gg n^{-d}$. It gives~\eqref{eq:X}, and hence, it completes our proof of the upper bound in Theorem~\ref{thm:small}.
\end{proof}


Next, we prove the upper bound in Theorem~\ref{thm:middle}
using Lemma~\ref{lem:Z_s+q+r(2)} as follows. 

\begin{proof}[\textbf{Proof of the upper bound in Theorem~\ref{thm:middle}}] For the upper bound, it suffices to show that there exists a positive constant $c=c(d)$ such that
\begin{equation}\label{eq:small prob} \Pr\Big[\mbox{$[n]^d_p$ contains a Sidon set of size $cn^{d/3}\(\log (n^{2d}p^3)\)^{1/3}$}\Big] \rightarrow 0 \hskip 0.5em \mbox{ as }  n\rightarrow \infty.
\end{equation}

The first moment method gives that the probability that $[n]^d_p$ contains a Sidon set of size $t$ is at most $p^t \cZ_{n,d}(t)$. We will use Lemma~\ref{lem:Z_s+q+r(2)} in order to bound $\cZ_{n,d}(t)$. Now we define suitable numbers $\gamma, \omega$, and $t$ satisfying both~\eqref{eq:q(2)} and~\eqref{eq:q(3)} in Lemma~\ref{lem:Z_s+q+r(2)}. For a positive constant $\delta\leq d/9$, we consider two cases separately: the first case is when $2n^{-2d/3}\leq p\leq n^{-2d/3+\delta}$ and the second case is for the remaining range of $p$, that is, $n^{-2d/3+\delta}\leq p\leq n^{-d/3-\epsilon}$.

\begin{itemize}
\item \textbf{Case 1:} This case is when $2n^{-2d/3}\leq p\leq n^{-2d/3+\delta}$. Let $\gamma=n^{2d}p^3$. Under the assumption $2n^{-2d/3}\leq p\leq n^{-2d/3+\delta}$, we have that $8\leq \gamma\leq n^{3\delta}\leq n^{d/3}$, and hence, the second inequality of~\eqref{eq:q(2)} holds.
Let $\displaystyle t=Cn^{d/3}\(\log(n^{2d}p^3)\)^{1/3}$, where $C=C(d)$ is a sufficiently large positive constant depending only on $d$. Then, the inequality of~\eqref{eq:q(3)} holds. 

With the choice of $\gamma$ and $t$, Lemma~\ref{lem:Z_s+q+r(2)} implies that
\begin{eqnarray}\label{eq:prob of Sidon}
\Pr\Big[\mbox{$[n]^d_p$ contains a Sidon set of size $t$}\Big] 
\leq  p^t\cZ_{n,d}(t)\leq \(\frac{4en^dp}{t\gamma^{1-2/\omega}}\)^t.
\end{eqnarray}
The base in the right-hand side of~\eqref{eq:prob of Sidon} is
\begin{eqnarray}\label{eq:base}
\frac{4en^dp}{t\gamma^{1-2/\omega}}&\leq&  \frac{4en^dp}{Cn^{d/3}(n^{2d}p^3)^{0.99}}\leq (n^{2d}p^3)^{1/3-0.99}\nonumber \\&\leq& 8^{1/3-0.99}\leq 0.5,
\end{eqnarray}
where the first inequality holds because of $t\geq Cn^{d/3}$, and the second inequalilty holds since $C\geq 4e$.
Combining~\eqref{eq:prob of Sidon} and~\eqref{eq:base} yields that
\begin{equation*}
\Pr\Big[\mbox{$[n]^d_p$ contains a Sidon set of size $t$}\Big] \leq 0.5^t \rightarrow 0 \hskip 0.5em \mbox{ as } n\rightarrow \infty,
\end{equation*}
which gives~\eqref{eq:small prob}.

\item \textbf{Case 2:} This case is when $n^{-2d/3+\delta}\leq p\leq n^{-d/3-\epsilon}$. Let $\gamma= n^{d/3}$. Then 
$$s^*=c_0 n^{d/3}\(\log(n^{d/3})\)^{1/3}=c'_0n^{d/3}(\log n)^{1/3}$$ with positive constants $c_0=c_0(d)$ and $c'_0=c'_0(d)$, and hence, the second inequality of~\eqref{eq:q(2)} holds.
Let $t=Cn^{d/3}\(\log(n^{2d}p^3)\)^{1/3}=C'n^{d/3}(\log n)^{1/3}$, where $C=C(d)$ and $C'=C'(d,\delta)$ are sufficiently large positive constants. Then, the inequality of~\eqref{eq:q(3)} holds.

With the choice of $\gamma$ and $t$, Lemma~\ref{lem:Z_s+q+r(2)} implies~\eqref{eq:prob of Sidon}.
 The base in the right-hand side of~\eqref{eq:prob of Sidon} is
\begin{eqnarray*}
\frac{4en^dp}{t\gamma^{1-2/\omega}}&\leq&  \frac{4en^d\cdot n^{-d/3-\epsilon}}{C'n^{d/3}(\log n)^{1/3} n^{(d/3)(1-\epsilon')}},
\end{eqnarray*}
where $\epsilon'=\epsilon'(d)$ is a positive constant such that $\epsilon'$ goes to $0$ as $C\rightarrow \infty$. We have that
\begin{eqnarray}\label{eq:base2}
\frac{4en^dp}{t\gamma^{1-2/\omega}} &\leq& \frac{n^{2d/3-\epsilon}}{n^{2d/3-d\epsilon'/3}\log n}\leq 0.5,
\end{eqnarray}
where the first and second inequalities follow from a choice of a sufficiently large $C=C(d)$.
Therefore, inequalities~\eqref{eq:prob of Sidon} and~\eqref{eq:base2} yield~\eqref{eq:small prob}.
\end{itemize}

Therefore, the analysis in \textbf{Case 1} and \textbf{Case 2} implies~\eqref{eq:small prob}, which completes our proof of the upper bound in Theorem~\ref{thm:middle}.
\end{proof}


Next we show the upper bounds in Theorems~\ref{thm:middle_large} and \ref{thm:large}. First, we claim that the upper bound in Theorem~\ref{thm:middle_large} follows from the upper bound in Theorem~\ref{thm:large}. Indeed,
by monotonicity, the upper bound in Theorem~\ref{thm:large} with $p=n^{-d/3}(\log n)^{8/3}$ gives the upper bound in Theorem~\ref{thm:middle_large}.
Therefore, it only remains to show the upper bound in Theorem~\ref{thm:large}. We show it by using Lemma~\ref{lem:Z_t} as follows.

\begin{proof}[\textbf{Proof of the upper bound in Theorem~\ref{thm:large}}] 

Let $q(t)$ be the probability that there exists a Sidon set in $[n]^d_p$ of size $t$. In order to show the upper bound in Theorem~\ref{thm:large}, it suffices to prove that there exists a positive constant $C=C(d)$ such that if $t=Cn^{d/2}p^{1/2}$, then $q(t)=o(1)$.

The first moment method gives that $q(t)\leq   p^t \cZ_{n,d}(t).$ Since $t=Cn^{d/2}p^{1/2}\geq Cn^{d/3}(\log n)^{4/3}$, Lemma~\ref{lem:Z_t} implies that
\begin{equation*}
q(t) \leq   n^{2(d+1)s_0}\(\frac{e2^{d+5}n^dp}{t^2}\)^t,
\end{equation*}
where $s_0=c n^{d/3}(\log n)^{1/3}$ with a positive constant $c=c(d)$. From the choice $t=Cn^{d/2}p^{1/2}$, we have that
\begin{equation*}
q(t)\leq n^{c'n^{d/3}(\log n)^{1/3}}\(\frac{e2^{d+5}}{C^2}\)^t\leq n^{c'n^{d/3}(\log n)^{1/3}} \(\frac{1}{2}\)^{t},
\end{equation*}
where $c'=c'(d)$ is a positive constant.
It is equivalent to the inequality
\begin{equation}
\log q(t)\leq c'n^{d/3}(\log n)^{4/3} + t\log\(1/2\).
\end{equation}
Since $t=Cn^{d/2}p^{1/2}\geq Cn^{d/3}(\log n)^{4/3}$ with a sufficiently large constant $C=C(d)$, we infer that
$ \log q(t)\leq  -2\log n,
$
that is, $q(t)\leq n^{-2}=o(1)$. This completes our proof of the upper bound in Theorem~\ref{thm:large}.
\end{proof}


\section{Lower bounds on $F([n]_p)$}\label{sec:lower bounds}

We are going to show the lower bounds in Theorems~\ref{thm:small}--\ref{thm:large}. To this end, we first introduce a result from~\cite{sidon, kohayakawa11*} about  lower bounds on the maximum size $F([n]_p)$  of Sidon sets in a random set $[n]_p=[n]^1_p$. Then, we define a bijection $\phi_d$ from $[n^d]$ to $[n]^d$, which was given by Cilleruelo~\cite{Cilleruelo2010}, such that a Sidon set in $[n^d]$ is mapped to a Sidon set in $[n]^d$. Using the bijection $\phi_d$, the lower bounds on $F([n^d]_p)$ in~\cite{sidon} will be converted to the lower bounds on $F([n]^d_p)$ in Theorems~\ref{thm:small}--\ref{thm:large}.

We first introduce the lower bounds on $F([n]_p)$ which were proved in Theorems 2.3--2.7 of~\cite{sidon}.

\begin{lemma}[\cite{sidon}]\label{lem:prob_n} There exist positive absolute constants $c_1$ and $c_2$ such that the following holds a.a.s.:

\begin{enumerate}

\item[(a)] $ F([n]_p) \geq (1+o(1))np $ \hskip 4.8em if \hskip 0.5em $n^{-1}\ll p\ll n^{-2/3}$,

\item[(b)]  $ F([n]_p)\geq \(1/3+o(1)\)np$ \hskip 3.7em if \hskip 0.5em $\displaystyle n^{-1}\ll p\leq2n^{-2/3}$,

\item[(c)]   $
    F([n]_p)\geq c_1 n^{1/3}\Big(\log (n^2p^3)\Big)^{1/3}$ \hskip 1em if \hskip 0.5em $2 n^{-2/3}\leq p\leq n^{-1/3}(\log n)^{2/3}$,

 \item[(d)] 
 $F([n]_p) \geq c_2\sqrt{np}$    \hskip 7.3em if \hskip 0.5em $n^{-1/3}(\log n)^{2/3}\leq p\leq 1.$
\end{enumerate}
\end{lemma}

Let $\phi_d: [n^d] \rightarrow [n]^d$ be the bijection defined by $\phi_d(a)=(a_0,\cdots, a_{d-1})$ where
$$ a=a_0+a_1n+a_2n^2+\cdots +a_{d-1}n^{d-1}.   $$
Cilleruelo~\cite{Cilleruelo2010} showed the following property of the bijection $\phi_d$. 

\begin{property}\label{prop:phi} If $A$ is a Sidon set in $[n^d]$, then $\phi_d(A)$ is a Sidon set in $[n]^d$. \end{property}

For a proof of Property~\ref{prop:phi}, see Theorem~5 and its proof in~\cite{Cilleruelo2010}.

 Now we are ready to show the following lower bounds on $F\([n]^d_p\)$ which easily imply the lower bounds in  Theorems~\ref{thm:small}--\ref{thm:large}.

\begin{lemma}\label{lem:lower bound}
 There exist positive absolute constants $c_1$ and $c_2$ such that the following holds a.a.s.:
\begin{enumerate}

\item[(a)] $ F\([n]^d_p\)\geq (1+o(1))n^dp $ \hskip 3.7em if \hskip 0.5em $n^{-d}\ll p\ll n^{-2d/3}$

\item[(b)]  $ F\([n]^d_p\)\geq \(1/3+o(1)\)n^dp$ \hskip 2.6em if \hskip 0.5em $ \displaystyle n^{-d}\ll p\leq2n^{-2d/3}$

 \item[(c)]    $
    F\([n]^d_p\)\geq c_1 n^{d/3}\Big(\log (n^{2d}p^3)\Big)^{1/3}$ \hskip 0em if \hskip 0.5em $2 n^{-2d/3}\leq p\leq d^{2/3}n^{-d/3}\(\log n\)^{2/3}$
 
 \item[(d)] 
 $F\([n]^d_p\) \geq c_2n^{d/2}p^{1/2}$    \hskip 5.2em if \hskip 0.5em $d^{2/3}n^{-d/3}\(\log n\)^{2/3}\leq p\leq 1.$
\end{enumerate}
\end{lemma}

\begin{proof}
Recall the bijection $\phi_d$ introduced just before Property~\ref{prop:phi}.
By the bijection $\phi_d$, a random set $[n^d]_p$ is mapped to $\phi_d\([n^d]_p\)$. Since $\phi_d\([n^d]_p\)$ is stochastically identical to $[n]^d_p$, we have that 
$[n^d]_p$ is stochastically identical to $[n]^d_p$.
Property~\ref{prop:phi} implies that  if $A\subset [n^d]_p$ is a Sidon set, then $\phi_d(A)\subset \phi_d([n^d]_p)=[n]^d_p$ is a Sidon set. Hence we infer that
  \begin{equation}\label{eq:lower bound} F([n^d]_p) \leq F\([n]^d_p\). \end{equation}

Therefore, in order to obtain a lower bound on $F\([n]^d_p\)$, one can use a lower bound on $F([n^d]_p)$. 
By Lemma~\ref{lem:prob_n} with $n^d$ instead of $n$, we obtain the following lower bounds on $F([n^d]_p)$: There exist absolute constants $c_1$ and $c_2$ such that the following holds a.a.s.:
\begin{enumerate}

\item[(a)] $ F([n^d]_p) \geq (1+o(1))n^dp $ \hskip 3.7em \textit{if} \hskip 0.5em $n^{-d}\ll p\ll n^{-2d/3}$

\item[(b)]  $ F([n^d]_p)\geq \(1/3+o(1)\)n^dp$ \hskip 2.5em \textit{if} \hskip 0.5em $\displaystyle n^{-d}\ll  p\leq2n^{-2d/3}$

\item[(c)]   $
    F([n^d]_p)\geq c_1 n^{d/3}\Big(\log (n^{2d}p^3)\Big)^{1/3}$ \hskip -0.2em \textit{if} \hskip 0.5em $2 n^{-2d/3}\leq p\leq n^{-d/3}\(\log (n^d)\)^{2/3}$
 
 \item[(d)] 
 $F([n^d]_p) \geq c_2\sqrt{n^dp}$    \hskip 6.1em \textit{if} \hskip 0.5em $n^{-d/3}\(\log (n^d)\)^{2/3}\leq p\leq 1.$
\end{enumerate}

\noindent Combining  inequality~\eqref{eq:lower bound} and the above (a)--(d) implies Lemma~\ref{lem:lower bound}.
\end{proof}

\begin{acknowledge}
  The author thanks
 Mark Siggers for helpful comments and corrections, and thanks Domingos Dellamonica Jr. for discussion yielding the improvement in Lemma~\ref{lem:Z_s+q+r(2)}. 
\end{acknowledge}



\providecommand{\bysame}{\leavevmode\hbox to3em{\hrulefill}\thinspace}
\providecommand{\MR}{\relax\ifhmode\unskip\space\fi MR }
\providecommand{\MRhref}[2]{%
  \href{http://www.ams.org/mathscinet-getitem?mr=#1}{#2}
}
\providecommand{\href}[2]{#2}
\def\MR#1{\relax}

\end{document}